\documentclass[12pt,reqno]{amsart}

\usepackage{amssymb, amsmath, amsthm, amsfonts, comment, subcaption}
\usepackage{bbm}
\usepackage[colorlinks=false]{hyperref}
\usepackage{tikz-cd}
\usetikzlibrary{matrix,arrows,decorations,decorations.markings,calc,positioning}
\usepackage{float}
\usepackage{setspace}

\newcommand{\arrow}[2][20]
 {
  \hspace{-5pt}
  \begin{tikzpicture}[scale=0.5]
   \node (A) at (0,0) {};
   \node (B) at (#1pt,0) {};
   \draw [#2,line width=0.5pt] (A) -- (B);
  \end{tikzpicture}
  \hspace{-5pt}
 }
\newcommand{\birational}[1][30]{\arrow[#1]{<->,dashed}}
\newcommand{\rational}[1][30]{\arrow[#1]{->,dashed}}

\tikzset{every picture/.style={scale=2,line width=0.7pt}}
\tikzstyle{longdashed}=[dashed,dash pattern=on 6pt off 6pt]

\def\formulaSpace{3pt}
\def\srational{\,\rational}
\def\soplus{\,\oplus\,}
\def\scong{\,\cong\,}
\def\seq{\,=\,}
\newcommand{\longto}{\longrightarrow}
\newcommand{\vecto}{\longrightarrow}

\newcommand\C{\mathbb C}

\newcommand\ul[1]{{#1}}
\newcommand\familyX{\ul{X}}
\newcommand\familyQ{\ul{Q}}

\newcommand\baseobj\cC
\newcommand\calcobj\cF
\newcommand\varobj\cB
\newcommand\sphobj\cE

\newcommand\res{\operatorname{res}}
\newcommand\id{\operatorname{id}}
\DeclareMathOperator \End {End}
\DeclareMathOperator \Hom {Hom}
\newcommand\sHom{\mathcal{H}om }
\newcommand\sEnd{\mathcal{E}nd }

\newcommand\LDerived{\mathrm{L}}\newcommand\RDerived{\mathrm{R}}
\DeclareMathOperator \RHom {\RDerived{}Hom}
\DeclareMathOperator \RsHom {\RDerived{}\mathcal{H}om}
\newcommand\Lotimes{\overset\LDerived\otimes}

\renewcommand\P{\mathbb P}

\newcommand{\Tw}{\operatorname{\fun[Tw]}}

\newcommand\cM{\mathcal M}
\newcommand{\cA}{\mathcal{A}}
\newcommand{\cB}{\mathcal{B}}
\newcommand{\cC}{\mathcal{C}}
\newcommand{\cE}{\mathcal{E}}
\newcommand{\cF}{\mathcal{F}}
\newcommand{\cN}{\mathcal{N}}

\newcommand{\cO}{\mathcal{O}}

\newcommand{\cT}{\mathcal{T}}
\newcommand{\cU}{\mathcal{U}}

\newcommand{\cR}{\mathcal{R}}

\newcommand{\Tot}{\operatorname{Tot}}
\newcommand{\Sym}{\operatorname{Sym}}

\newcommand{\comp}{\circ}

\newlength\tempWidth

\newcommand\Cone{\mathrm{Cone}}
\newcommand\D{\mathsf{D}}
\newcommand\Exc{\mathrm{Exc}}

\newcommand\fun[1][F]{\operatorname{\mathsf{#1}}}
\newcommand\lab[1]{#1 : \quad}
\newcommand\labpos[1]{#1 \quad\quad}

\newcounter{keyenvcount}

\numberwithin{equation}{section}

\theoremstyle{plain}
\newtheorem{thm}{Theorem}[section]
\newtheorem{keythm}[keyenvcount]{Theorem}
\newtheorem{prop}[thm]{Proposition}

\theoremstyle{remark}
\newtheorem{rem}[thm]{Remark}
\newtheorem*{keyrem}{Remark}
\newtheorem*{acks}{Acknowledgements}

\theoremstyle{definition}
\newtheorem{defn}[thm]{Definition}

\newtheorem{notn}[thm]{Notation}

\newcommand{\marginparstretch}{0.6}
\let\oldmarginpar\marginpar
\renewcommand\marginpar[1]{\-\oldmarginpar[\framebox{\setstretch{\marginparstretch}\begin{minipage}{\marginparwidth}{\raggedleft\scriptsize #1}\end{minipage}}]{\framebox{\setstretch{\marginparstretch}\begin{minipage}{\marginparwidth}{\raggedright\scriptsize #1}\end{minipage}}}}

\def\trianglepicscaleA{0.9}
\def\trianglepicscaleB{1.0}

\newcommand\trianglepic[1][0]{
\def\birat{2} \def\onemorph{1} \def\twomorphintro{0} \def\twoopts{3} \def\twomorphtext{4} \def\twomorphtextother{5} \def\biratB{6}
\def\plainintro{7}

\def\arg{#1}

\ifx\arg\biratB\def\posa{-0.9,-0.3}\else\def\posa{-1.3,0}\fi
\ifx\arg\biratB\def\posb{-0.25,0.7}\else\def\posb{0,0.65}\fi
\ifx\arg\biratB\def\posc{0.1,-0.65}\else\def\posc{0,-0.65}\fi
\def\posd{0.5,0.1}

\def\birattri{
\node (0) at (\posa) {$X_1$};
\node (1) at (\posb) {$X_2$};
\node (2) at (\posc) {$X_3$};
\ifx\arg\biratB\node (3) at (\posd) {$X_4$};\fi
\draw[<->,longdashed] (0) -- (1);
\draw[<->,longdashed] (1) -- (2);
\draw[<->,longdashed] (2) -- (0);
\ifx\arg\biratB
\draw[<->,longdashed] (0) -- (3);
\draw[<->,longdashed] (1) -- (3);
\draw[<->,longdashed] (2) -- (3);
\fi
}

\def\derivtri{
\node (0) at (\posa) {$\D(X_1)$};
\node (1) at (\posb) {$\D(X_2)$};
\node (2) at (\posc) {$\D(X_3)$};
}

\ifx\arg\birat\birattri\fi
\ifx\arg\biratB\birattri\fi

\ifx\arg\plainintro\derivtri\fi
\ifx\arg\onemorph\derivtri\fi
\ifx\arg\twomorphintro\derivtri\fi
\ifx\arg\twoopts\derivtri\fi
\ifx\arg\twomorphtext\derivtri\fi
\ifx\arg\twomorphtextother\derivtri\fi

\def\drawA{\draw[->]}
\def\drawB{\draw[->]}
\def\drawC{\draw[->]}

\ifx\arg\onemorph
\def\drawA{\draw[->,transform canvas={xshift=-1pt,yshift=3pt}]}
\def\drawAr{\draw[<-,transform canvas={xshift=1pt,yshift=-3pt}]}
\def\drawB{\draw[->,transform canvas={xshift=3pt,yshift=0pt}]}
\def\drawBr{\draw[<-,transform canvas={xshift=-3pt,yshift=0pt}]}
\def\drawC{\draw[<-,transform canvas={xshift=-1pt,yshift=-3pt}]}
\def\drawCr{\draw[->,transform canvas={xshift=1pt,yshift=3pt}]}
\fi

\def\biratarrows{
\draw[<->,longdashed] (0) -- (1);
\draw[<->,longdashed] (1) -- (2);
\draw[<->,longdashed] (2) -- (0);
}

\ifx\arg\birat\biratarrows\fi
\ifx\arg\biratB\biratarrows\fi

\ifx\arg\birat\else\ifx\arg\biratB\else

\drawA (0) -- 
   node[pos=0.55,above left] {\ifx\arg\twoopts $\fun_{21}$ \fi \ifx\arg\twomorphtextother $\fun_{21}$ \fi \ifx\arg\twomorphtext $\fun_{21}$ \fi \ifx\arg\twomorphintro $\fun_1$ \fi \ifx\arg\plainintro $\fun_1$ \fi}  (1);

\ifx\arg\twoopts\else
\ifx\arg\twomorphtextother
\drawB (2) -- 
   node[right] {$\fun_{23}$} (1);
\else
\drawB (1) -- 
   node[right] {\ifx\arg\twomorphtextother $\fun_{32}$ \fi \ifx\arg\twomorphtext $\fun_{32}$ \fi \ifx\arg\twomorphintro $\fun_2$ \fi \ifx\arg\plainintro $\fun_2$ \fi} (2);
\fi
\fi

\drawC (0) -- 
   node[pos=0.55,below left] {\ifx\arg\twoopts $\fun_{31}$ \fi \ifx\arg\twomorphtext $\fun_{31}$ \fi \ifx\arg\twomorphtextother $\fun_{31}$ \fi \ifx\arg\twomorphintro $\fun_3$ \fi \ifx\arg\plainintro $\fun_3$ \fi} (2);

\ifx\arg\onemorph
\drawAr (0) -- (1);
\drawBr (1) -- (2);
\drawCr (0) -- (2);
\fi

\fi\fi

\def\rel{
\draw[->,dotted] ($(0,-0.5)+(100:0.45)$) arc (100:160:0.45);
\node (t) at (-0.40,0.04) {$\Tw_2$};
}
\def\relother{
\draw[->,dotted] ($(0,0.5)+(-100:0.45)$) arc (-100:-160:0.45);
\node (t) at (-0.38,+0.02) {$\Tw_3$};
}

\ifx\arg\twomorphintro\rel\fi
\ifx\arg\twomorphtext\rel\fi
\ifx\arg\twomorphtextother\relother\fi
}

\begin{document}

\title[Derived equivalences for simplices of higher-dimensional flops]{Relating derived equivalences for \\ simplices of higher-dimensional flops}

\author{W.\ Donovan}
\address{W.\ Donovan, Yau Mathematical Sciences Center, Tsinghua University, Haidian District, Beijing 100084, China.}
\email{donovan@mail.tsinghua.edu.cn}

\thanks{I am supported by the Yau MSC, Tsinghua University, and the Thousand Talents Plan. I am grateful for travel support from EPSRC Programme Grant EP/R034826/1}

\begin{abstract} 
I study a sequence of singularities in dimension~$4$ and above, each given by a cone of rank~$1$ tensors of a certain signature, which have crepant resolutions whose exceptional loci are isomorphic to cartesian powers of the projective line. In each dimension~$n$, these resolutions naturally correspond to vertices of an $(n-2)$-simplex, and flops between them correspond to edges of the simplex. I show that each face of the simplex may then be associated to a certain relation between flop functors.
\end{abstract}



\maketitle
\setcounter{tocdepth}{1}


\section{Introduction}

This note is motivated by a desire to better understand the derived autoequivalence groups of $4$-folds and beyond, in particular the contributions coming from birational geometry. I study triangles of birational maps between resolutions of certain singularities, and find that the corresponding flop functors obey a pleasing relation involving spherical twists, extending a well-known story for $3$-folds related by Atiyah flops. 

\subsection{Singularities and resolutions} Consider the $n$-fold singular cone for $n\geq 3$ given by the rank~$1$ tensors of signature $2^{n-1}$ as follows.
\[ Z = \{ v_1 \otimes \dots \otimes v_{n-1} \in V_1 \otimes \dots \otimes V_{n-1} \} \qquad V_i \cong \C^2 \]
By a straightforward construction explained later, $Z$ has $n-1$ crepant resolutions. These will be written $X_i$ for $i=1,\dots,n-1$. Each is given by replacing the singularity $0 \in Z$ by an exceptional locus $\Exc \cong (\P^1)^{n-2}$ which, for a given $i$, arises as a product of the $\P V_j$ for $j\neq i$. 

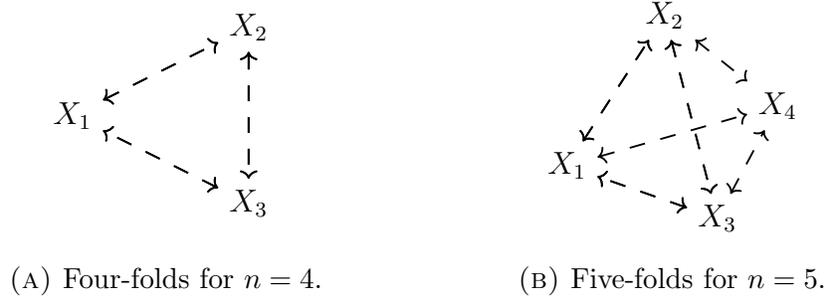
\begin{figure}[h]
\begin{minipage}[b]{.45\linewidth}
\begin{center}
\begin{tikzpicture}[scale=\trianglepicscaleA] \trianglepic[2] \end{tikzpicture}
\vspace{0.2cm}
\end{center}
\subcaption{Four-folds for $n=4$.}
\end{minipage}
\begin{minipage}[b]{.45\linewidth}
\begin{center}
\begin{tikzpicture}
\trianglepic[6]
\end{tikzpicture}
\end{center}
\subcaption{Five-folds for $n=5$.}
\end{minipage}
\caption{Resolutions and birational maps for small $n$.}
\label{functors4I}
\end{figure}

For $n=3$, we have the $3$-fold Atiyah flop between two resolutions of the cone of singular $2\times 2$ matrices. In general, assigning each resolution $X_i$ to a vertex of an $(n-2)$-simplex, the edges of the simplex correspond to birational maps $X_i \birational X_j $ as illustrated in Figure~\ref{functors4I}. For~$n=4$ we have a~triangle of $4$-folds, and for~$n=5$ we have a~tetrahedron of $5$-folds.

\subsection{Equivalences} The birational maps appearing here are family Atiyah flops, and therefore have associated flop functors, which are derived equivalences, illustrated in Figure~\ref{functors4J}.

\begin{figure}[h]
\begin{center}
\begin{tikzpicture}[scale=\trianglepicscaleB] \trianglepic[1] \end{tikzpicture}
\end{center}
\caption{Flop functors for $n=4$.}
\label{functors4J}
\end{figure}
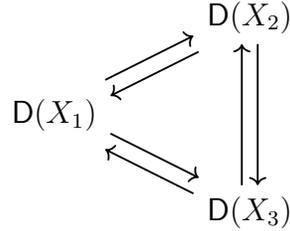

This work grew out of interest in describing `derived monodromy' for \mbox{Figure~\ref{functors4J}}, namely an autoequivalence given by composition of a $3$-cycle of equivalences in the triangle. However, it seems that a question with a neater answer is the following.

\begin{center}\it How are routes between two of the vertices of Figure~\ref{functors4J} related?\end{center}

Theorem~\ref{mainthm.equiv} below gives such a relation where we take, without loss of generality, the two vertices $X_1$ and $X_3$. Furthermore it gives a relation associated to each face ($2$-simplex) in the analogous diagram for $n>4$. 

\subsection{Result} I prove the following.

\begin{keythm}[Theorems~\ref{thm.keynatiso} and~\ref{thm.equiv}]\label{mainthm.equiv} For $n\geq 4$, write flop functors as follows.
\smallskip
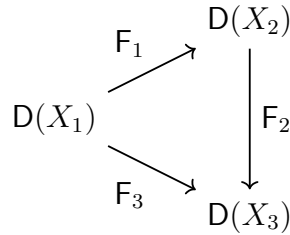
\begin{center}
\begin{tikzpicture}[scale=\trianglepicscaleB]
\trianglepic[7]
\end{tikzpicture}
\end{center} 
Then there is a natural isomorphism
\[  \fun_3 \cong \Tw_2 \comp \fun_2 \comp \fun_1 \]
where $\Tw_2$ is one of the following.
\begin{description}
\item[Case $n=4$] a spherical twist around the torsion sheaf $\cO_\Exc (0,-1)$ on $X_3$, where $\Exc \cong \P V_1 \times \P V_2$, given in Definition~\ref{defn.twist}.
\item[Case $n>4$] a family version of this spherical twist, given in Definition~\ref{defn.twistfam}, over base $\P V_4 \times \dots \times \P V_{n-1}$.
\end{description}
\end{keythm}

\noindent For each $n$, replacing the indices $1,2,3$ in the above with general $i,j,k$ gives a relation for each face ($2$-simplex) of the $(n-2)$-simplex.

\begin{keyrem} As a mnemonic for Theorem~\ref{mainthm.equiv}, I draw a diagram as follows.
\smallskip
\begin{center}
\begin{tikzpicture}[scale=\trianglepicscaleB]
\trianglepic
\end{tikzpicture}
\end{center}
\end{keyrem}

I prove Theorem~\ref{mainthm.equiv} by calculating the action of flops and twists on a certain (relative) tilting bundle.

\subsection{Related questions}

As an immediate corollary  of Theorem~\ref{mainthm.equiv}, we get the following formula for~$\Tw_2$.
\begin{equation*}\tag{$\ast$}\label{equation flops} \fun_3 \comp \fun_1^{-1} \comp \fun_2^{-1} \cong \Tw_2
\end{equation*}
There are many formulas of the form `flop-flop = twist' (up to taking inverses) in the literature, including~\cite{ADM,ALb,Bar,BB,DS1,DW1,DW3,JL,Har,Tod}. I~hope~\eqref{equation flops} gives a hint of how they may be generalized to flop cycles of length~3 and above.

For discussion of the relation of Theorem~\ref{mainthm.equiv} to derived monodromy, in particular to the autoequivalence of $\D(X_1)$ given by a composition of $3$~flop functors, see Remark~\ref{rem.monod}.

\subsection{Atiyah flop}\label{sect.Atiyah} For $n=3$, we have resolutions $X_1$ and $X_2$ related by an Atiyah flop, and functors as follows, satisfying the relation shown, where $\Tw$ is a spherical twist about $\cO_\Exc(-1)$.
\begin{center}
\begin{tikzpicture}[scale=1.5]
\node (0) at (-1,0) {$\D(X_1)$};
\node (1) at (0,0)  {$\D(X_2)$};
\node (2) at (1.3,0) {$\id \cong \Tw \circ  \fun_2 \comp \fun_1$};
\draw[->,transform canvas={yshift=+2.5pt}] (0) -- node[above]{$\fun_1$} (1);
\draw[<-,transform canvas={yshift=-2.5pt}] (0) -- node[below]{$\fun_2$} (1);
\end{tikzpicture}
\end{center}
The argument for Theorem~\ref{mainthm.equiv} is an elaboration of a standard argument for this relation: see Section~\ref{sec.discuss} for discussion.

\subsection{Related work} The example $n=4$ is studied extensively by Kite~\cite[Sections~5.3 and~7.2]{Kit}. He gives an action of $\pi_1(\cM)$ on the derived categories~$\D(X_i)$ where $\cM$ is a certain Fayet--Iliopoulos parameter space. I expect Theorem~\ref{mainthm.equiv} in the case $n=4$ also follows from his methods.

Kite uses a realization of the $X_i$ as toric GIT quotients, and the technology of `fractional magic windows'. This may give an alternative, and perhaps swifter, method to prove Theorem~\ref{mainthm.equiv}. However, as it was not needed to obtain the result, I feel the proofs here may be more accessible without it. I also have in mind extensions to more general flops, where a GIT presentation is not available.

Halpern-Leistner and Sam~\cite{HLSam} construct similar actions of $\pi_1$ for GIT problems which are `quasi-symmetric'. This condition does not hold for the GIT quotients realizing the $X_i$.

\begin{keyrem} The constructions here readily generalize to the case of $V_i \cong \C^d$ for any fixed $d>2$, and I expect that similar results can be proved, see Remark~\ref{rem.gen}. Kite notes that, for $n=4$, the above `fractional magic windows' technology may not apply in this new setting~\cite[Example~4.37]{Kit}.\end{keyrem}

\subsection{Contents}Section~\ref{sect.res} describes the resolutions~$X_i$ and flop functors between them. Section~\ref{sect.flopcalc} gives properties of these functors for later use. Sections~\ref{sec.4} and~\ref{sect.higher} prove Theorem~\ref{mainthm.equiv}, first in dimension~$4$ and then in higher dimension by a family construction. Section~\ref{sect.fam} gives further family constructions relating resolutions.

\subsection{Notation} When I write $X^{(n)}$ and similar notations, the $(n)$ denotes the dimension~$n$ of the space, or the relative dimension~$n$ of a family. Letters L and~R indicate derived functors throughout, but are sometimes dropped in Section~\ref{sect.higher} for the sake of readability. The bounded derived category of coherent sheaves is denoted by~$\D(X)$.

\begin{acks}I am grateful for conversations with Tatsuki Kuwagaki, Xun Lin, Mauricio Romo, Ed Segal, and Weilin Su, and thank the organizers of the conference `McKay correspondence, mutation and related topics' at Kavli IPMU for their work to make the meeting a success during the pandemic. I started calculations for this project while visiting Michael Wemyss at University of Glasgow, and am grateful for his hospitality and support there.
\end{acks}

\section{Resolutions}\label{sect.res}

We construct resolutions of the singularity $Z$ from the introduction, namely
\[ Z = \{ v_1 \otimes \dots \otimes v_{n-1} \in V_1 \otimes \dots \otimes V_{n-1} \} \qquad V_i \cong \C^2  \]
for $n\geq 3$. To see that $\dim Z = n$, note that $z\in Z - \{ 0 \}$ taken up to scale determines a point of the cartesian power $(\P^1)^{n-1}$. There are various terminologies for~$Z$. For instance, we may describe it as the cone of rank~$1$ tensors of signature~$2^{n-1}$, or as the simple hypermatrices of order $n-1$ in dimension $2$.

\subsection{Construction} We construct crepant resolutions $X_i$ of $Z$ for $i=1,\dots,n-1$.

\begin{defn} Let $X_1$ be the total space of a rank~$2$ bundle
\[ \lab{X_1} V_1 \otimes \cO(-1,\dots,-1)  \vecto \P V_2 \times \dots \times \P V_{n-1} \]
where the line bundle $\cO(-1,\dots,-1)$ has degree~$-1$ on each factor $\P V_i$. Other spaces $X_i$ are obtained similarly, by applying the cyclic symmetry of the set $\{ V_i \}$ to replace $V_1$ by $V_i$.\end{defn}

\begin{notn} Indices will be written in cyclic order. For instance, for $n=4$ I~write the following.
\begin{align*} 
\lab{X_1} & V_1 \otimes \cO(-1,-1)  \vecto \P V_2 \times \P V_3 \\
\lab{X_2} & V_2 \otimes \cO(-1,-1)  \vecto \P V_3 \times \P V_1 \\
\lab{X_3} & V_3 \otimes \cO(-1,-1)  \vecto \P V_1 \times \P V_2
\end{align*}
\end{notn}

To see the resolution morphism $g_1 \colon X_1 \to Z$ write $X_1$ as follows, where $L_i$ denotes the tautological subspace bundle on $\P V_i$.
\[ \lab{X_1} V_1 \otimes (L_2 \boxtimes \dots \boxtimes L_{n-1}) \vecto \P V_2 \times \dots \times \P V_{n-1} \]
Then the inclusions $L_i \subset V_i$ induce the required morphism, which is easily seen  to contract the zero section $\Exc_1 \subset X_1$ while being an isomorphism elsewhere. The other $g_i$ are obtained similarly.

\begin{notn} Let $\Exc_i \subset X_i$ be the exceptional locus of the resolution morphism $g_i \colon X_i \to Z$. I often drop subscripts and write $\Exc$ for readability.\end{notn}

By the following proposition, the $g_i$ are crepant resolutions.

\begin{prop}\label{prop.cy} Each $X_i$ is Calabi--Yau.\end{prop}
\begin{proof} The total space of the bundle $X_1$ is Calabi--Yau because both its determinant, and the canonical bundle of its base $\P V_2 \times \dots \times \P V_{n-1}$, are isomorphic to $\cO(-2,\dots,-2)$. The same holds for the other $X_i$ by the cyclic symmetry.\end{proof}

I set notation for sheaves and bundles on the $X_i$ for $n=4$.

\begin{notn}\label{sect bun notation} For $X_1$, let $\cO_\Exc(a,b)$ denote a line bundle on $\Exc_1 \cong \P V_2 \times \P V_3$ considered as a torsion sheaf on $X_1$. Let $\cO(a,b)$ denote a line bundle on  $X_1$ given by pullback from the base $\P V_2 \times \P V_3$. Similar notations are used for the other $X_i$. 
\end{notn}

\subsection{Flop functors}
\label{sect flop fun}

If we blow up the zero section $\Exc_i \subset X_i$ for any $i$ we obtain
\[ \lab{Q} \cO(-1,\dots,-1)  \vecto \P V_1 \times \dots \times \P V_{n-1} \]
with a blowup map $f_i \colon Q \to X_i$. Using these $f_i$ to form birational roofs, we have birational maps
\[ \phi_{ji} \colon X_i \srational X_j. \]

\begin{defn}\label{def.ffunc} We write functors
\[ \fun_{ji} = \RDerived f_{j*} \circ \LDerived f_i^* \colon \D(X_i) \longrightarrow \D(X_j). \]
\end{defn}

These functors are equivalences. For $n=3$ they are simply the Bondal--Orlov equivalences for the Atiyah flop. For $n\geq 4$, the $\phi_{ji}$ are family Atiyah flops which implies that the $\fun_{ji}$ are equivalences. This is explained in Section~\ref{sect.fam_descrip} for $n=4$, and a similar argument using Proposition~\ref{prop.gen_fam} suffices for general $n$.

\begin{rem} For $n=4$ the birational maps $\phi_{ji}$ may be drawn as follows. Note that at each $X_i$ the two birational maps have the same exceptional locus $\P^1 \times \P^1$, but that each is a flop of a different ruling.
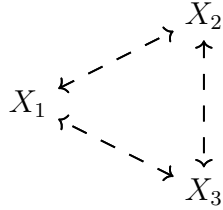
\begin{center}
\begin{tikzpicture}[scale=\trianglepicscaleA] 
\trianglepic[2]
\end{tikzpicture}
\end{center}
\end{rem}

\begin{rem}\label{rem.gen} The constructions in this section extend to the setting where $V_i \cong \C^d$ for any fixed $d \geq 2$. The flops appearing are then families of standard flops of~$\P^{k-1}$. It would be interesting to prove similar results in this case. For comparison, see \cite{ADM} for a `flop-flop = twist' formula for such flops.
\end{rem}

\subsection{Notation for families}
\label{sect notat}

For a given $n$, each resolution may be constructed as a non-trivial family of the analogous resolution for $n-1$. I explain the $n=4$ case in the following Section~\ref{sect.flopcalc}. The general case, which is analogous but notationally more complex, is deferred to Proposition~\ref{prop.iterate}. Here I give the notation that will be used in these constructions.

\smallskip

To specify a particular $n$, the notation~$X_i^{(n)}$ is used. To specify furthermore the vector spaces used in the construction, I write the following.
\[X_i^{(n)}(V_1,\dots,V_{n-1})\]
The construction will often be repeated in a family, replacing the vector spaces $V_i$ with vector bundles $\pi_i$ over some base $B$. The result is denoted as follows. 
\[\familyX_i^{(n)}(\pi_1,\dots,\pi_{n-1})\]
Finally, similar notations are used for other constructions, for instance the birational roof~$Q$ from Section~\ref{sect flop fun} may be written as $Q^{(n)}(V_1,\dots,V_{n-1}).$

\section{Flop calculations}\label{sect.flopcalc}

I explain how each resolution for $n=4$ may be constructed as a non-trivial family of the analogous resolutions for $n=3$, and use this to calculate the effect of the flop functors for $n=4$ on certain objects, for use in the proof of Theorem~\ref{thm.keynatiso}. This calculation is routine, but I write it out in full to show how the non-triviality is handled, and anticipating a further family version in Section~\ref{sect.higher}.

\subsection{Family construction}\label{sect.fam_descrip}

Recall that we take the following.
\[\lab{X_1^{(4)}} V_1 \otimes \cO(-1,-1)  \vecto \P V_2 \times \P V_3\]
The proposition below realizes $X_1^{(4)}$ as a family  of copies of $X_1^{(3)}$ over  $\P V_3$.

\smallskip

\begin{prop}\label{prop.fam4} Using the notation of Section~\ref{sect notat}, we have that
\begin{equation}\label{eqn.family1} X_1^{(4)}(V_1, V_2, V_3) \cong \familyX_1^{(3)} (\pi_1, \pi_2) \end{equation}
where we take bundles
\begin{align*} 
\lab{\pi_1} & V_1\otimes\cO \vecto  \P V_3 \\
\lab{\pi_2} & V_2\otimes\cO(-1) \vecto  \P V_3
\end{align*}
so that $\pi_1$ is the trivial bundle with fibre $V_1$. 
\end{prop}
\begin{proof}
Writing $ \Tot (\P \pi_2) $ for the total space of $\P \pi_2$, there is an isomorphism
\[ \Tot (\P \pi_2) \cong \P V_2 \times \P V_3 \] under which $\cO_{\P \pi_2}(-1)$ corresponds to $\cO(-1,-1)$, where the second $-1$ comes from the definition of~$\pi_2$, giving the claim.
\end{proof}

We now extend the argument to birational roofs. We have an isomorphism 
\begin{equation}\label{eqn.family2} X_2^{(4)}(V_1, V_2, V_3) \cong \familyX_2^{(3)} (\pi_1, \pi_2) \end{equation}
using that $ \Tot (\P \pi_1) \cong \P V_3 \times \P V_1$ under which $\cO_{\P \pi_1}(-1)$ corresponds to $\cO(0,-1)$.
Furthermore, we have an isomorphism
\begin{equation}\label{eqn.family3}Q^{(4)}(V_1, V_2, V_3) \cong \familyQ^{(3)}(\pi_1, \pi_2)\makebox[0pt]{ .} 
\end{equation}
The isomorphisms \eqref{eqn.family1}, \eqref{eqn.family2} and \eqref{eqn.family3} intertwine birational roof diagrams of blowup maps $f_i \colon Q \to X_i$ as follows.
\begin{center}
\begin{tikzpicture}[scale=0.6,xscale=1.7]
\node (14) at (0,0) {$X_1^{(4)}(V_1, V_2, V_3) $};
\node (r4) at (1,1) {$Q^{(4)}(V_1, V_2, V_3) $};
\node (24) at (2,0) {$X_2^{(4)}(V_1, V_2, V_3) $};
\draw[->] (r4) -- (14);
\draw[->] (r4) -- (24);
\end{tikzpicture}
\qquad
\begin{tikzpicture}[scale=0.6,xscale=1.5]
\node (13) at (0,0) {$\familyX_1^{(3)} (\pi_1, \pi_2)$};
\node (r3) at (1,1) {$Q^{(3)}  (\pi_1, \pi_2)$};
\node (23) at (2,0) {$\familyX_2^{(3)} (\pi_1, \pi_2)$};
\draw[->] (r3) -- (13);
\draw[->] (r3) -- (23);
\end{tikzpicture}
\end{center}
The right-hand diagram is a family of $3$-fold Atiyah flops. It follows that the flop functor $\fun_{21}$ from Definition~\ref{def.ffunc} using the left-hand diagram is an equivalence.

\subsection{Flop functors}\label{section.flop}

The above construction lets us calculate the effect of the flop $X_1 \birational X_2$ for $n=4$ on the derived category. Recall that $\cO(a,b)$ denotes a certain line bundle on each of the $X_i$ by the convention of Notation~\ref{sect bun notation}.

\begin{prop}\label{prop flop action general} For $n=4$, the flop functor
\[ \fun \colon \D(X_1) \to \D(X_2) \]
acts as follows, where we write projections $\rho_i\colon X_i \to \P V_3$.

\begin{enumerate}
\item\label{prop flop action general a} For any $\varobj \in \D(\P V_3)$, taking $\cA_1,\cA_0 \in \D(X_1)$ given by
\begin{align*}
\cA_1 &= \rho_1^* \varobj \otimes \cO(-1,0) \\
\cA_0 &= \rho_1^* \varobj
\end{align*}
we have the following.
\begin{align*}
\fun(\cA_1) & \cong \rho_2^* \varobj \otimes \cO(1,1)  \\
\fun(\cA_0) & \cong \rho_2^* \varobj
\end{align*}

\item\label{prop flop action general b} There exist canonical inclusions, given at the end of the proof, such that the following diagram commutes, with $\Hom$s taken in the derived category.
\begin{center}
\begin{tikzpicture}[scale=0.8,xscale=1.3]

\node (1) at (-1,0) {$\Hom(\cA_1,\cA_0)$};
\node (2) at (1,0)  {$\Hom(\fun(\cA_1),\fun(\cA_0))$};
\node (1b) at (-0.05,-1) {$V_2^\vee$};

\draw[->] (1) to node[above]{\scriptsize $\fun$} (2);
\draw[left hook->] (1b) to (1);
\draw[right hook->] (1b) to (2);

\end{tikzpicture}
\end{center}
\end{enumerate}
\end{prop}
\begin{proof} The $4$-fold $X_1$ is isomorphic to the family  of $3$-folds $X^{(3)}_1(\pi_1, \pi_2)$ over $\P V_3$ by Proposition~\ref{prop.fam4}. Under this isomorphism $\cA_1,\cA_0 \in \D(X_1)$ go to
\begin{equation}\label{eq.flop_inter1}  \rho_1^* ( \varobj \otimes \cO(1) ) \otimes \cO_{\P \pi_2}(-1) \quad\text{and}\quad \rho_1^* \varobj \end{equation} where we reuse $\rho_1$ for the projection from $X^{(3)}_1(\pi_1, \pi_2)$ to $ \P V_3$. The twist $\otimes \cO(1)$ appearing here is dual to the twist in the definition of $\pi_2$. Then the flop 
\begin{equation}\label{eqn.familyflop} X^{(3)}_1(\pi_1, \pi_2) \birational X^{(3)}_2(\pi_1, \pi_2)\end{equation} is a family over $\P V_3$ of copies of the $3$-fold Atiyah flop $Y_1  \birational Y_2$  with
\begin{align*}
\lab{Y_1} & W_1 \otimes \cO(-1)  \vecto \P W_2 \\
\lab{Y_2} & W_2 \otimes \cO(-1) \vecto \P W_1
\end{align*} 
where $\dim W_i = 2$. Write $\fun[G]\colon \D(Y_1)  \to \D(Y_2)$ for the flop functor.

The following description of the effect of $\fun[G]$ is obtained by standard arguments, see for instance~\cite[Proposition~1]{DS2}. Letting
\[\cB_1 = \cO(-1) \qquad \cB_0 = \cO\] where $\cO(k)$ denotes a bundle on $Y_i$ obtained by pullback from $\P W_i$, we have
\[\fun[G](\cB_1) = \cO(+1) \qquad \fun[G](\cB_0) = \cO\makebox[0pt]{ .}\]
Furthermore we have a commutative triangle
\begin{center}
\begin{tikzpicture}[scale=0.8,xscale=1.3]
\node (1) at (-1,0) {$\Hom(\cB_1,\cB_0)$};
\node (2) at (1,0)  {$\Hom(\fun[G](\cB_1),\fun[G](\cB_0))$};
\node (1b) at (-0.05,-1) {$W_2^\vee$};
\draw[->] (1) to node[above]{\scriptsize $\fun[G]$} (2);
\draw[left hook->] (1b) to (1);
\draw[right hook->] (1b) to (2);
\end{tikzpicture}
\end{center}
where the inclusions are given by observing the following.
\begin{align*}
\Hom(\cB_1,\cB_0) \cong \RDerived\Gamma_{Y_1} \,\cO(+1) & \cong \RDerived\Gamma_{\P W_2} \big( \cO(+1) \otimes \Sym^\bullet ( W_1^\vee \otimes \cO(1)) \big) \\
\Hom(\fun[G](\cB_1),\fun[G](\cB_0)) \cong \RDerived\Gamma_{Y_2} \,\cO(-1) & \cong \RDerived\Gamma_{\P W_1} \big( \cO(-1) \otimes \Sym^\bullet ( W_2^\vee \otimes \cO(1)) \big)
\end{align*}
These have no higher cohomology by standard vanishing on $\P^1$, and after taking $0^{\text{th}}$ cohomology, the pieces coming from $\Sym^0$ and $\Sym^1$ respectively are both $W_2^\vee$.

Now repeating the standard arguments in a family, we find that the flop functor for the flop \eqref{eqn.familyflop} applied to the objects~\eqref{eq.flop_inter1} gives
\begin{equation}\label{eq.flop_inter2} \rho_2^* ( \varobj \otimes \cO(1) ) \otimes \cO_{\P \pi_1}(+1) \quad\text{and}\quad \rho_2^* \varobj  \end{equation}
where we reuse $\rho_2$ for the projection  from $X^{(3)}_2(\pi_1, \pi_2)$ to $ \P V_3$. Under the isomorphism~\eqref{eqn.family2} it is easily seen that the objects~\eqref{eq.flop_inter2} go to the required objects on~$X_2$.

For the last part, we also repeat the argument for the Atiyah flop in a family. The inclusions in the statement are given by observing the following.
\begin{align*}
\Hom(\cA_1,\cA_0) & \cong \RDerived\Gamma_{X_1}\, \cO(1,0) \\
& \cong  \RDerived\Gamma_{\P V_2 \times \P V_3} \big( \cO(1,0) \otimes \Sym^\bullet ( V_1^\vee \otimes \cO(1,1)) \big) \\[\formulaSpace]
\Hom(\fun(\cA_1),\fun(\cA_0)) & \cong \RDerived\Gamma_{X_2}\, \cO(-1,-1) \\
& \cong \RDerived\Gamma_{\P V_3 \times \P V_1} \big( \cO(-1,-1) \otimes \Sym^\bullet ( V_2^\vee \otimes \cO(1,1)) \big)
\end{align*}
These have no higher cohomology by standard vanishing on $\P^1 \times \P^1$, and after taking $0^{\text{th}}$ cohomology, the pieces coming from $\Sym^0$ and $\Sym^1$ respectively are both $V_2^\vee$.
\end{proof}

The following describes the action of the flop functors on certain line bundles, again using the conventions of Notation~\ref{sect bun notation}.

\begin{prop}\label{prop flop action} For $n=4$, considering the flop functors
\begin{center}
\begin{tikzpicture}[scale=\trianglepicscaleB]
\trianglepic[3]
\end{tikzpicture}
\end{center}
between
\begin{align*} 
\lab{X_1} & V_1 \otimes \cO(-1,-1)  \vecto \P V_2 \times \P V_3 \\
\lab{X_2} & V_2 \otimes \cO(-1,-1)  \vecto \P V_3 \times \P V_1 \\
\lab{X_3} & V_3 \otimes \cO(-1,-1)  \vecto \P V_1 \times \P V_2
\end{align*}
we have the following.
\begin{align*} 
\lab{\fun_{21}} \cO(0,b) & \mapsto \cO(b,0) \\
\cO(-1,b) & \mapsto \cO(b+1,1) \\[\formulaSpace]
\lab{\fun_{31}} \cO(a,0) & \mapsto \cO(0,a) \\
\cO(a,-1) & \mapsto \cO(1,a+1)
\end{align*}
\end{prop}
\begin{rem}\label{rem cycle} From this proposition we may deduce the action of all $\fun_{ji}$ by cyclic symmetry. Because of our conventions, the statements for $\fun_{32}$ and $\fun_{13}$ read the same as for  $\fun_{21}$, and so on.
\end{rem}
\begin{proof}
The claim for $\fun_{21}$ is from Proposition~\ref{prop flop action general} using isomorphisms as follows, and the claim for $\fun_{31}$ is obtained by symmetry.
\[
\cO(0,b) \cong \rho_1^* \cO(b) \quad\text{and}\quad
\cO(-1,b) \cong \rho_1^* \cO(b) \otimes \cO(-1,0)\quad\text{on~$X_1$}
\]
\[
\cO(b,0) \cong \rho_2^* \cO(b) \quad\text{and}\quad \cO(b+1,1) \cong \rho_2^* \cO(b) \otimes \cO(1,1)\quad\text{on~$X_2$}\qedhere
\]
\end{proof}

\begin{rem} Propositions~\ref{prop flop action general} and~\ref{prop flop action} may also be obtained directly by adapting an argument of Kawamata~\cite[Proposition~3.1]{Kaw}, after calculating the canonical bundles of the roofs $Q$.
\end{rem}

The following straightforward observation about the flop functors~$\fun_{ji}$ will be used in the proof of Theorem~\ref{thm.keynatiso}.
  
\begin{prop}\label{prop.intertwine} Writing restriction functors 
\[ \res_i \colon \D(X_i) \to \D(X_i - \Exc_i) \]
we have intertwinements
\begin{align*}
\res_j \comp \fun_{ji} & \cong g_{ji*} \comp \res_i
\end{align*}
where we write $g_{ji}\colon X_i - \Exc_i \to X_j - \Exc_j$ for the isomorphism induced by the birational map $\phi_{ji}$.
\end{prop}
\begin{proof}Each $\res_i$ is by definition a pullback along an open immersion, so this follows using flat base~change.
\end{proof}

\section{Dimension~$4$}\label{sec.4}

Here I prove Theorem~\ref{mainthm.equiv} for dimension~$4$. Though the argument is standard, I~write it in detail, anticipating a family version in Section~\ref{sect.higher}. After the proof, some discussion of the method is given in Section~\ref{sec.discuss}.

\subsection{Proof} Recall that $\Exc \cong \P V_1 \times \P V_2$ is the exceptional locus of $X_3$.

\begin{prop}\label{prop.sph} The object $\cE=\cO_\Exc(0,-1)$
is spherical in $\D(X_3)$.
\end{prop}
\begin{proof} Noting that $X_3$ is Calabi--Yau by Proposition~\ref{prop.cy}, we require 
\[ \Hom(\cE,\cE) \cong \operatorname{H}^*(S^4) \]
which follows by a standard spectral sequence calculation, using that for the normal bundle $\cN$ of $\Exc$ we have
\begin{align*}
\cN & \cong \cO_{\P^1\times \P^1}(-1,-1)^{\oplus 2} \\
\wedge^2 \cN & \cong \cO_{\P^1\times \P^1}(-2,-2) .
\end{align*}
 See for instance~\cite[Examples~8.10(v)]{Huy} for a similar calculation.
\end{proof}

\begin{defn}\label{defn.twist} Take the spherical twist autoequivalence on $X_3$
\[ \Tw_2 = \Tw \big( \cO_\Exc(0,-1) \big) \]
where the subscript~$2$ is used because pullback of $\cO_{\P V_2}(-1)$ gives $\cO_\Exc(0,-1)$. The autoequivalence $\Tw (\cE)$ is defined so that there is a triangle of Fourier--Mukai functors
\[ \Tw (\cE) = \Cone\big(  \RHom_{X_3}(\cE,-)\Lotimes \cE  \to \id \big)\makebox[0pt]{\,,} \]
see \cite{ST} or~\cite[Section~8.1]{Huy}.
\end{defn}

Similarly, we have an autoequivalence on $X_2$ \[ \Tw_3 = \Tw \big( \cO_\Exc(-1,0) \big) \]
where the subscript~$3$ is used because pullback of $\cO(-1)_{\P V_3}$ gives $\cO_\Exc(-1,0)$.

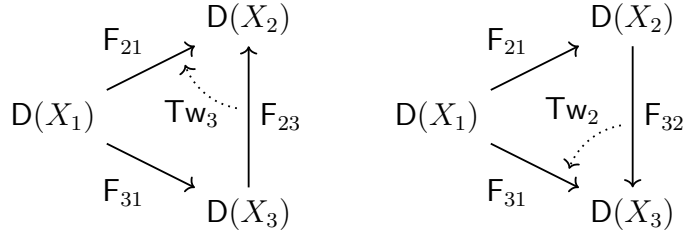
\begin{figure}[tbh]
\begin{center}
\begin{tikzpicture}[scale=\trianglepicscaleB]
\trianglepic[5]
\end{tikzpicture}
\qquad
\begin{tikzpicture}[scale=\trianglepicscaleB]
\trianglepic[4]
\end{tikzpicture}
\end{center}
\caption{Relations between functors from Theorem~\ref{thm.keynatiso}.}
\label{functors.rel}
\end{figure}

\begin{thm}[Theorem~\ref{mainthm.equiv}]\label{thm.keynatiso}  For $n=4$ there are natural isomorphisms
\begin{align*}
\fun_{21} & \cong \Tw_3 \comp \fun_{23} \comp \fun_{31}  \\
\fun_{31} & \cong \Tw_2 \comp \fun_{32} \comp \fun_{21}
\end{align*}
of functors from $\D(X_1)$ to $\D(X_2)$ and $\D(X_3)$ respectively, illustrated in Figure~\ref{functors.rel}.
\end{thm}

\begin{proof}
I prove the second statement, as the first follows by the same argument. Using the results of Section~\ref{section.flop}, we evaluate the functors in the proposed isomorphism on the following bundle $\cT$ on $X_1$. Note that $\cT$ is tilting, by standard methods using the Beilinson tilting bundles $\cO \oplus \cO(-1)$ on $\P V_2$ and $\P V_3$.
\[ \cT = \cO \soplus \cO(-1,0) \soplus \cO(0,-1) \soplus \cO(-1,-1) \]
By applying functors to the summands using Proposition~\ref{prop flop action}, we find that
\[
\setlength\arraycolsep{2pt}
\begin{array}{rcccccccc}
\fun_{21} (\cT) & \cong & \cO & \oplus & \cO(1,1) & \oplus & \cO(-1,0) & \oplus & \cO(0,1) \\[\formulaSpace]
\fun_{31} (\cT) & \cong & \cO & \oplus & \cO(0,-1) & \oplus & \cO(1,1) & \oplus & \cO(1,0) 
\end{array}
\]
and furthermore by Remark~\ref{rem cycle} that
\[
\fun_{32}\fun_{21} (\cT) \scong \cO  \soplus  \calcobj  \soplus  \cO(1,1)  \soplus  \cO(1,0)\makebox[0pt]{ ,}
\]
where we let
\begin{equation}\label{eqn cE}
\calcobj = \fun_{32} (\cO(1,1)) \cong \fun_{32}\fun_{21} (\cO(-1,0))\makebox[0pt]{ .}
\end{equation}
This $\calcobj$ may be described as follows.
\begin{align*}
\calcobj & \cong \fun_{32} \Cone \big( \!\wedge^2 \! V_3^\vee \otimes  \cO(-1,1) \overset{\psi_2\,}\longto V_3^\vee \otimes \cO(0,1) \big) \\
& \cong \Cone \big( \!\wedge^2 \! V_3^\vee \otimes  \cO(2,1) \overset{\psi_3\,}\longto V_3^\vee \otimes \cO(1,0) \big)
\end{align*}
The first line uses the pullback of the Euler short exact sequence from $\P V_3 \cong \P^1$ to~$X_2$. The second line follows by Proposition~\ref{prop flop action}, where we let $\psi_3 = \fun_{32} (\psi_2)$.

To describe $\psi_3$ we apply Proposition~\ref{prop flop action general}(\ref{prop flop action general b}) with $\cB=\cO(1)$ on $\P V_1$ and cycling the indices as in Remark~\ref{rem cycle}. As in this proposition, $\Hom (\cO(-1,1),\cO(0,1))$ on $X_2$ has a canonical summand $V_3^\vee$. By construction, $\psi_2$ is induced by this summand. If follows from the proposition that $\psi_3$ is induced by the canonical summand~$V_3^\vee$ of $\Hom (\cO(2,1),\cO(1,0))$ on $X_3$. Using the Koszul resolution of the exceptional locus $\Exc$ on $X_3$, we therefore find the following.
 \begin{align}\label{eqn.inv_tw_sh}
\calcobj & \cong \Cone \big(  \cO(0,-1) \overset{\operatorname{res}}{\longrightarrow} \cO_\Exc(0,-1) \big)[-1] \notag  \\
& \cong \Tw_2^{-1} (\cO(0,-1)) \notag \\
& \cong \Tw_2^{-1} \fun_{31} (\cO(-1,0))
\end{align}
For the second isomorphism we use that 
\begin{equation}\label{eqn.invtwist} \Tw_2^{-1} = \Tw (\cE)^{-1} \cong \Cone\big(\id \to \RHom_{X_3}(-,\cE)^\vee \Lotimes \cE  \,\big)[-1] \end{equation}
with $\cE=\cO_\Exc(0,-1)$, where the morphism is an adjunction unit and
\begin{equation}\label{eqn.onkeysheaf} \RHom_{X_3}(\cO(0,-1),\cO_\Exc(0,-1))^\vee \cong \RDerived\Gamma_\Exc (\cO_\Exc)^\vee = \C \end{equation}
by projectivity.

I now argue that we have the following.
 \begin{equation}\label{eq.alt_iso}  \Tw_2^{-1}  \fun_{31} (\cT) \cong \fun_{32} \fun_{21} (\cT)\end{equation} 
We first take a splitting $\cT=\cU \oplus \cF$ with $\cU$ given below, and show that \eqref{eq.alt_iso} holds with $\cT$ replaced by $\cU$.
\[ \cU \seq \cO \soplus \cO(0,-1) \soplus \cO(-1,-1) \]
By the above argument $ \fun_{31} (\cU) \cong \fun_{32} \fun_{21} (\cU) $ with 
\[ \fun_{31} (\cU) \scong \cO \soplus \cO(1,1) \soplus \cO(1,0). \]
This is in the kernel of
\[ \RHom_{X_3}(-,\cO_\Exc(0,-1)), \] and therefore is unchanged up to isomorphism by applying $\Tw_2^{-1}$. We deduce that \eqref{eq.alt_iso} holds with $\cT$ replaced by $\cU$. Combining with the definition~\eqref{eqn cE} and description~\eqref{eqn.inv_tw_sh} of $\calcobj$, we conclude~\eqref{eq.alt_iso}.

Finally, we prove the claim by considering the `difference' of the two sides of the claimed isomorphism, namely proving the following natural isomorphism of functors on $\D(X_1)$.
\begin{equation}\label{eq.diff_iso}   \Psi = \fun_{21}^{-1} \comp \fun_{32}^{-1} \comp \Tw_2^{-1} \comp \fun_{31} \cong \id \end{equation} 
By~\eqref{eq.alt_iso}, $\Psi(\cT) \cong \cT$, so $\Psi$ induces a composition as follows.
\begin{equation}\label{eq.tilt_alg} \End (\cT) \overset\sim\longrightarrow \End (\Psi(\cT)) \cong \End (\cT) \end{equation}
We will show this is the identity, and thence that \eqref{eq.diff_iso} holds by the tilting equivalence. 

Recall the restriction functors $\res_i \colon \D(X_i) \to \D(X_i - \Exc)$ of Proposition~\ref{prop.intertwine}. We have \[ \res_1 \circ \, \Psi  \cong \res_1 \] by combining the intertwinements of Proposition~\ref{prop.intertwine} with 
\begin{equation}\label{eq.sph_intertw}\res_2 \comp \Tw_2^{-1} \cong \res_2\end{equation}
which follows from definition of the twist, in particular that the spherical object is supported on $\Exc$. It follows immediately that \eqref{eq.tilt_alg} intertwines via $\res_1$ with the identity on $\End (\res_1 \cT)$. Noting that $\res_1 \cT$ is just the bundle $\cT|_{X_1 - \Exc}$, that $\Exc$ is codimension~$2$, and $X_1$ is smooth thence normal, we deduce that \eqref{eq.tilt_alg} is the identity. Using that $\cT$ is tilting so that there is an equivalence $\D(\End (\cT)) \cong \D(X_1)$, we find that \eqref{eq.diff_iso} holds, and this completes the proof.\end{proof}

\begin{rem}\label{rem.monod} I briefly explain how Theorem~\ref{thm.keynatiso} above relates to calculating the derived monodromy around the triangle formed by the $\D(X_i)$, namely, to determining the composition $\fun_{13} \comp \fun_{32} \comp \fun_{21}$. Using the theorem we have the following.
\begin{align*}
\fun_{13} \comp \fun_{32} \comp \fun_{21} & \cong \fun_{13}  \comp \Tw_2^{-1}  \comp \fun_{31} \\
 & \cong (\fun_{13} \comp \fun_{31})  \comp (\fun_{31}^{-1}  \comp \Tw_2^{-1}  \comp \fun_{31})
 \end{align*}
The two brackets may then be calculated by standard techniques. The first may be expressed as a product of twists of spherical objects by, for instance, flop-flop formulas for toric variation of GIT in~\cite{HLShi}. The second may be expressed as a twist by a spherical object using the following.
\[ \Phi \comp \Tw(\cE) \comp \Phi^{-1} \cong \Tw(\Phi \cE) \quad \Longleftrightarrow \quad \Phi \comp \Tw^{-1}(\cE) \comp \Phi^{-1} \cong \Tw^{-1}(\Phi \cE) \] 
It would be interesting to carry out this calculation, for this example and more generally.
\end{rem}

\def\xpadB{0.45}
\def\xpad{-0.45}
\def\ypadB{0.13}
\def\ypad{0.15}
\def\tilt{-0.05}
\newcommand\nodeloc[4]{($#1*(0:1)+#2*(120:1)+(#3,#4)$)}

\def\tiltingB
{
\draw[lightgray] (-1.1-\xpadB,\ypadB) -- (-1.1-\xpadB,-\ypadB) -- (\xpadB,-\ypadB) -- (\xpadB,\ypadB) -- cycle;
}

\def\tilting
{
\draw[lightgray] (-\xpad+\tilt,\ypad) -- \nodeloc{0}{-1}{-\xpad-\tilt}{-\ypad} -- \nodeloc{-1}{-1}{\xpad-\tilt}{-\ypad} -- \nodeloc{-1}{0}{\xpad+\tilt}{\ypad} -- cycle;
}

\def\nodes
{
\node (Am1m1) at \nodeloc{-1}{-1}{0}{0} {$\hspace{12pt}\cO(-1,-1)$};
\node (Am10) at \nodeloc{-1}{0}{0}{0} {$\hspace{4pt}\cO(-1,0)$};
\node (A0m1) at \nodeloc{0}{-1}{0}{0} {$\cO(0,-1)$};
\node (A10) at \nodeloc{1}{0}{0}{0} {$\cO(1,0)$};
\node (A01) at \nodeloc{0}{1}{0}{0} {$\cO(0,1)$};
\node (A11) at \nodeloc{1}{1}{0}{0} {$\cO(1,1)$};
\node (A00) at (0,0) {$\cO$};
}

\def\nodesB
{
\node (Bm1) at (-1,0) {$\cO(-1)$};
\node (B0) at (0,0) {$\cO$};
\node (B1) at (1,0) {$\cO(1)$};
}

\subsection{Discussion}\label{sec.discuss} For some intuition for the above Theorem~\ref{thm.keynatiso}, I include Figure~\ref{functors4} showing the action of the flop functors on line bundles on the $X_i$. The arrows in Figure~\ref{functors4} indicate source and target for the given functor, up to isomorphism. The dotted arrows indicate the same thing, but where the target $\cO(a,b)$ should be replaced with \[\Tw_{\cO_\Exc(a,b)}^{-1} \cO(a,b) \cong \Cone\big(\cO(a,b) \to \cO_\Exc(a,b)\big)[-1] \cong \mathcal{I}_\Exc(a,b) \]
using a similar argument to the proof of Theorem~\ref{thm.keynatiso}. Note also that each flop functor takes $\cO$ to~$\cO$. Therefore in the diagrams each flop ``cycles" the bundles by $2\pi / 3$, up to spherical twists.

Inspecting Figure~\ref{functors4} we see, for instance, that $\fun_{32} \comp \fun_{21}$ and $\fun_{31}$ give the same results on all summands of the tilting bundle $\cT$ except for $\cO(-1,0)$. The spherical twist $\Tw_2$ in Theorem~\ref{thm.keynatiso} accounts precisely for the disparity.

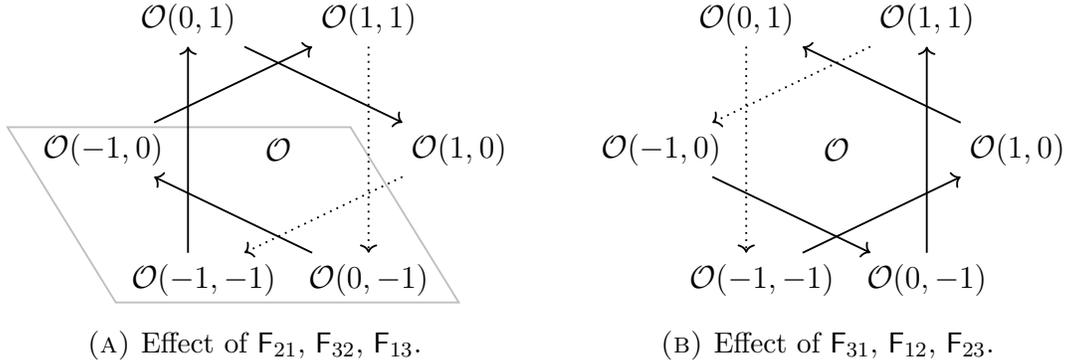
\begin{figure}[tbh]
\begin{minipage}[b]{.45\linewidth}
\begin{center}
\begin{tikzpicture}[xscale=1.2]
\tilting
\nodes
\draw[->] (Am1m1) to (A01);
\draw[->] (Am10) to (A11);
\draw[->] (A0m1) to (Am10);
\draw[->]  (A01) to (A10);
\draw[dotted,->]  (A11) to (A0m1);
\draw[dotted,->]  (A10) to (Am1m1);
\end{tikzpicture}
\end{center}
\subcaption{Effect of $\fun_{21}$, $\fun_{32}$, $\fun_{13}$.}
\end{minipage}
\hspace{0.75cm}
\begin{minipage}[b]{.45\linewidth}
\begin{center}
\begin{tikzpicture}[xscale=1.2]
\nodes
\draw[->] (Am1m1) to (A10);
\draw[->] (A0m1) to (A11);
\draw[->] (Am10) to (A0m1);
\draw[->]  (A10) to (A01);
\draw[dotted,->]  (A11) to (Am10);
\draw[dotted,->]  (A01) to (Am1m1);
\end{tikzpicture}
\end{center}
\subcaption{Effect of $\fun_{31}$, $\fun_{12}$, $\fun_{23}$.}
\end{minipage}
\caption{Flop functors for $n=4$, with summands of $\cT$ highlighted.}
\label{functors4}
\end{figure}

\begin{rem} There is a similar diagram, albeit simpler and more well-known, for the case $n=3$, given in Figure~\ref{functors3}. Here we have $X_1$ and $X_2$ related by an Atiyah flop, and functors as follows.
\begin{center}
\begin{tikzpicture}[scale=1.5]
\node (0) at (-1,0) {$\D(X_1)$};
\node (1) at (0,0)  {$\D(X_2)$};
\draw[->,transform canvas={yshift=+2.5pt}] (0) -- node[above]{$\fun_{21}$} (1);
\draw[<-,transform canvas={yshift=-2.5pt}] (0) -- node[below]{$\fun_{12}$} (1);
\end{tikzpicture}
\end{center}

Then Figure~\ref{functors3} shows the action of the flop functors. The dotted arrow indicates that the target $\cO(a)$ should be replaced with \[\Tw_{\cO_\Exc(a)}^{-1} \cO(a) \cong \Cone\big(\cO(a) \to \cO_\Exc(a)\big) [-1] \cong \mathcal{I}_\Exc(a) \]
by standard arguments. This can be used to prove the relation $\id \cong \Tw \circ  \fun_2 \comp \fun_1$ from Section~\ref{sect.Atiyah} for the Atiyah flop, by a simple analogue of the argument of Theorem~\ref{thm.keynatiso}. Namely, we check the relation on the tilting bundle $\cO(-1)\oplus\cO$, and then deduce that it holds in general.

\begin{figure}[thb] 
\begin{center}
\begin{tikzpicture}[xscale=1]
\tiltingB
\nodesB
\draw[->,bend left,looseness=1] (Bm1) to (B1);
\draw[dotted,->,bend left,looseness=1]  (B1) to (Bm1);
\end{tikzpicture}
\end{center}
\caption{Flop functors $\fun_{21}$, $\fun_{12}$ for $n=3$, with tilting bundle.}
\label{functors3}
\end{figure}
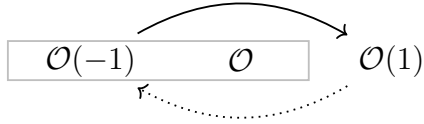

\end{rem}

\section{Higher dimension} \label{sect.higher}

Having proved the result for $n=4$ in Section~\ref{sec.4}, I explain how to extend to $n>4$ by a family construction.

\subsection{Family construction}

I realize $X_1^{(n)}$ as a family of $X_1^{(4)}$ over base 
\[ B = \P V_4 \times \dots \times \P V_{n-1}\makebox[0pt]{ .} \]

\begin{prop}\label{prop.famhigher} For $n>4$ we have for $i=1,2,3$
\[ X_i^{(n)}(V_1, \dots, V_{n-1}) \cong X_i^{(4)} (\pi_1, \pi_2, \pi_3 ) \]
where we take the following bundles.
\begin{align*} 
\lab{\pi_1} & V_1\otimes\cO \vecto B \\
\lab{\pi_2} & V_2\otimes\cO \vecto B \\
\lab{\pi_3} & V_3\otimes\cO(-1,\dots,-1) \vecto B
\end{align*}
\end{prop}
\begin{proof} Recalling that we have
\[ \lab{X_1^{(n)}} V_1 \otimes \cO(-1,-1,-1,\dots,-1)  \vecto \P V_2 \times \P V_3 \times \P V_4 \times \dots \times \P V_{n-1} \]
the result follows by the methods of Section~\ref{sect.fam_descrip}.
\end{proof}

Proposition~\ref{prop.famhigher} also holds trivially for $n=4$, if we take $B$ to be a point.

\begin{rem} The isomorphisms of Proposition~\ref{prop.famhigher} express each $n$-fold as a family of $4$-folds, similarly to how the isomorphisms~\eqref{eqn.family1} and~\eqref{eqn.family2} expressed $4$-folds as a family of~$3$-folds.
\end{rem}

\subsection{Family spherical twist}\label{sect.famtwist}
I construct a family spherical twist on~$X_3$ for $n\geq 4$ generalizing the twist~$\Tw_2$ for $n=4$ from Definition~\ref{defn.twist}. Via the isomorphism of Proposition~\ref{prop.famhigher}, the exceptional locus on $X_3$ is 
\[\Exc = \P \pi_1 \underset{B}{\times} \P \pi_2 ,\]
 a bundle over $B$ with fibre $\P V_1 \times \P V_2$. Write \[\cO_\Exc(a,b) = \cO_{\P \pi_1}(c) \underset{B}{\boxtimes} \cO_{\P \pi_2}(d)\] and consider this as a torsion sheaf on $X_3$ via the inclusion. This is a relative analog of Notation~\ref{sect bun notation}. Continuing the analogy, write $\cO(a,b)$ for the pullback of $\cO_\Exc(a,b)$ to $X_3$, and use similar notation for the other $X_i$. 
 
\smallskip

I define the following functor, which I will show to be spherical. Here and in the next subsection, the letters L and R on derived functors are often dropped for the sake of readability.

\begin{defn} Take the functor
\begin{align*} 
S & = \sphobj \otimes \tau^*(-) \colon \D(B) \to \D(X_3)
\end{align*}
where $\sphobj=\cO_\Exc(0,-1)$ on $X_3$ and $\tau \colon X_3 \to B$ denotes the projection morphism. 
\end{defn}

\begin{prop}\label{prop.adj} We have adjoints to $S$ as follows.
\begin{align*} 
L &  = \tau_! \sHom ( \sphobj, - ) \cong \tau_* \sHom ( -, \sphobj )^\vee \\
R & = \tau_* \sHom ( \sphobj, - )
\end{align*}
\end{prop}
\begin{proof} First note that $\sphobj^\vee \otimes -$, where we take derived dual, is a two-sided adjoint to  $\sphobj \otimes - $. We then use that $\tau_! \dashv \tau^* \dashv \tau_*$. The isomorphism follows after noting that $\tau_! \cong \tau_*(-^\vee)^\vee$.
\end{proof}

\begin{prop}
$S$ is a spherical functor.
\end{prop}
\begin{proof} 
Using the framework of Anno--Logvinenko~\cite{AL}, $S$ yields a cotwist endofunctor $C$ of $\D(B)$, satisfying the following.
\[ C \cong \Cone(\id \to RS)[-1]\]
It suffices that this is an equivalence, along with a certain Calabi--Yau condition, namely that a canonical natural transformation $R\to CL[1]$ is an isomorphism.

Now we have that 
\begin{align*}
RS & \cong \tau_* \sHom ( \sphobj, \sphobj \otimes \tau^*(-) ) \\
& \cong \tau_* \big( \sHom ( \sphobj, \sphobj) \otimes \tau^*(-) \big)  \\
& \cong \tau_* \sHom ( \sphobj, \sphobj) \otimes -    
\end{align*}
where all functors are derived. Using a family version of the argument of Proposition~\ref{prop.sph}, we may then calculate $\tau_* \sHom ( \sphobj, \sphobj)$ and obtain an isomorphism
\[RS \cong \id \oplus\, (-\otimes\omega_B) [-4]\]
where $\omega_B \cong \cO(-2,\dots,-2)$. The $\omega_B$ in this calculation arises from the determinant of the relative normal bundle of $\Exc$ over $B$, namely 
\[\wedge^2 \cN_{\text{rel}}  \cong \cO_{\Exc}(-2,-2) \otimes \sigma^*\omega_B \]
where we take projection $\sigma\colon \Exc \to B$. It follows that the cotwist
\[ C \cong  (-\otimes\omega_B)[-5]\]
which is an autoequivalence.

I briefly explain how the Calabi--Yau condition mentioned above is obtained from the Calabi--Yau property of the target space $X_3$. Note that for the fibration $\tau\colon X_3 \to B$ we have $\omega_\tau \cong \tau^* \omega_B$ because $X_3$ is Calabi--Yau by Proposition~\ref{prop.cy}. Then recall
\[ \tau_! = \tau_* ( - \otimes \omega_\tau [\dim \tau]) \cong  \tau_* (-) \otimes \omega_B  [4] \] where we use the projection formula. The condition then follows using Proposition~\ref{prop.adj}.
\end{proof}

\begin{defn}\label{defn.twistfam} Take the spherical twist autoequivalence on $X_3$
\begin{equation*} \Tw_2 = \Tw (S) \end{equation*}
where the subscript 2 is used because under the isomorphism $\Exc \cong \P V_1  \times \P V_2 \times B$ the bundle $\cO_\Exc(0,-1)$ on $\Exc$ appearing in the definition of $S$ is the pullback of~$\cO_{\P V_2}(-1)$.
The autoequivalence $\Tw (S)$ is defined so that there is a triangle of Fourier--Mukai functors
\begin{equation*} \Tw (S) \cong \Cone(SR \to \id)\makebox[0pt]{\,,} \end{equation*}
as explained in~\cite{AL}.
\end{defn}

Similarly, we have an autoequivalence $\Tw_3$ on $X_2$, by repeating the construction of this subsection using instead  $\sphobj=\cO_\Exc(-1,0)$ on $X_2$ and projection $\tau \colon X_2 \to B$.

\begin{rem} Definition~\ref{defn.twistfam} reduces to Definition~\ref{defn.twist} when $n=4$, using Proposition~\ref{prop.adj}.
\end{rem}

\begin{rem} The above twists can be formulated as EZ-twists \cite{Hor}, see for instance \cite[Definition~8.43]{Huy}. 
Indeed, taking projection $\sigma\colon \Exc \to B$ and inclusion $i\colon \Exc \to X_3$ we get that
 \begin{align*} S & \cong i_* (\sphobj \otimes \sigma^* (-))
\end{align*}
using $\tau \comp i = \sigma$ and the projection formula.
\end{rem}

\subsection{Proof} The following uses a family version of the argument of Theorem~\ref{thm.keynatiso} to generalize the result there to dimension $n> 4$.

\begin{thm}[Theorem~\ref{mainthm.equiv}]\label{thm.equiv}  For $n\geq 4$ there are natural isomorphisms
\begin{align*}
\fun_{21} & \cong \Tw_3 \comp \fun_{23} \comp \fun_{31}  \\
\fun_{31} & \cong \Tw_2 \comp \fun_{32} \comp \fun_{21}
\end{align*}
of functors from $\D(X_1)$ to $\D(X_2)$ and $\D(X_3)$ respectively, where the twists~$\Tw$ are defined in Definition~\ref{defn.twistfam} and following it.
\end{thm}
\begin{proof}
As before, we prove the second statement, with the first following by the same argument. We replace the spaces $X_i$ for $i=1,2,3$ in the proof of Theorem~\ref{thm.keynatiso} with
\[ \familyX_i^{(4)} (\pi_1,\pi_2, \pi_3) ,\]
writing projections $\tau_i\colon X_i^{(4)} \to B$, or simply $\tau$. The proof proceeds by repeating the argument of Theorem~\ref{thm.keynatiso} in a family over $B$. I explain the key modifications needed for this relative context.

The tilting bundle $\cT$ of Theorem~\ref{thm.keynatiso}, namely
\[ \cT = \cO \soplus \cO(-1,0) \soplus \cO(0,-1) \soplus \cO(-1,-1) \]
makes sense in the relative context using the notation of Section~\ref{sect.famtwist}. It is now a relative tilting bundle over $B$, as follows. Taking $\cR = \tau_* \sEnd(\cT)$, a sheaf of algebras on $B$, we have
\[ \D(\cR) \cong \D(X_1) \]
where we take right $\cR$-modules, with mutually inverse equivalences as follows.
\[\RDerived\tau_* \RsHom(\cT,-) \qquad\qquad \tau^{-1}(-) \underset{\tau^{-1}\cR}{\Lotimes} \cT\]

The description of the flop functors is similar to Propositions~\ref{prop flop action general} and~\ref{prop flop action}. We replace the $\cA_i$ as written there with 
$\cA_i(\baseobj)$ for any $\baseobj\in \D(B)$, where we put 
\[ \cA(\baseobj) = \tau^{-1}(\baseobj) \underset{\tau^{-1} \cO_B}\otimes \cA. \]
In place of the commuting diagram in Proposition~\ref{prop flop action general}(\ref{prop flop action general b}) it suffices to prove the result with the following diagram.
\begin{center}
\begin{tikzpicture}[scale=0.8,xscale=1.8]

\node (1) at (-1,0) {$\tau_*\sHom(\cA_1,\cA_0)$};
\node (2) at (1,0)  {$\tau_*\sHom(\fun(\cA_1),\fun(\cA_0))$};
\node (1b) at (-0.05,-1) {$V_2^\vee\otimes\cO $};

\draw[->] (1) to node[above]{\scriptsize $\fun$} (2);
\draw[left hook->] (1b) to (1);
\draw[right hook->] (1b) to (2);

\end{tikzpicture}
\end{center}

Given this, the calculation of the action of the flop functors on~$\cT$ at the beginning of the proof of Theorem~\ref{thm.keynatiso} proceeds in the relative context. Note that when Proposition~\ref{prop flop action general}(\ref{prop flop action general b}) is applied to describe $\psi_3$, it takes the following form, with the twist $\cO(1,\dots,1)$ being dual to the twist in the definition of $\pi_3$.
\begin{center}
\begin{tikzpicture}[scale=0.8,xscale=1.9]

\node (1) at (-1,0) {$\tau_*\sHom(\cA_1,\cA_0)$};
\node (2) at (1,0)  {$\tau_*\sHom(\fun(\cA_1),\fun(\cA_0))$};
\node (1b) at (-0.05,-1) {$V_3^\vee\otimes\cO(1,\dots,1) $};

\draw[->] (1) to node[above]{\scriptsize $\fun$} (2);
\draw[left hook->] (1b) to (1);
\draw[right hook->] (1b) to (2);

\end{tikzpicture}
\end{center}

For equation \eqref{eqn.inv_tw_sh} in the relative context we have
 \begin{align}\label{eqn.inv_tw_sh.rel}
\calcobj & = \fun_{32}\fun_{21} (\cO(-1,0)) \notag \\ 
& \cong \Cone \big(  \cO(0,-1) \overset{\operatorname{res}}{\longrightarrow} \cO_\Exc(0,-1) \big)[-1] \notag \\
& \cong \Tw_2^{-1}  (\cO(0,-1)) \notag \\
& \cong \Tw_2^{-1} \fun_{31} (\cO(-1,0))
\end{align}
where the second isomorphism arises as follows. The formula~\eqref{eqn.invtwist} for the inverse twist is replaced by 
\begin{equation*}\label{eqn.invtwistfam} \Tw_2^{-1} \cong \Cone(\id \to SL)[-1] \end{equation*}
and \eqref{eqn.onkeysheaf} is replaced by 
\begin{align*}\label{eqn.onkeysheaffam} L(\cO(0,-1))
& \cong \tau_* \sHom ( \cO(0,-1), \sphobj)^\vee \\
& \cong \tau_* \sHom ( \cO(0,-1), \cO_\Exc(0,-1))^\vee \\
& \cong \sigma_* (\cO_\Exc)^\vee \\
& \cong \cO_B. \end{align*}
Here $\sigma\colon \Exc \to B$ denotes the projection morphism, and we obtain the last line using that $\sigma$ is a bundle with fibre $\P V_1 \times \P V_2$.  Finally, we use $S (\cO_B) \cong \sphobj = \cO_\Exc(0,-1) $ to get \eqref{eqn.inv_tw_sh.rel}.

The analogs of the intertwinements of Proposition~\ref{prop.intertwine} and \eqref{eq.sph_intertw} follow by similar arguments, using that all the Fourier--Mukai functors are relative to $B$, that is their kernels are pushed forward from the fibre product over $B$.

The argument concludes by showing the following.
\begin{equation*} \Psi = \fun_{21}^{-1} \comp \fun_{32}^{-1} \comp \Tw_2^{-1} \comp \fun_{31} \cong \id \end{equation*} 
For this, we study the endomorphism  induced by $\Psi$ of the sheaf of algebras $\cR = \tau_* \sEnd(\cT)$, similarly to the end of the proof of Theorem~\ref{thm.keynatiso}. 
\end{proof}

\begin{rem} It would be interesting to give a global analog of Theorem~\ref{mainthm.equiv} taking, for instance, a collection of quasiprojective $Y_i$ having a diagram of birational maps as for the $X_i$, and further having formal completions isomorphic to the formal completions of the $X_i$ along $\Exc_i$. A first step could be to extend existing methods for proving relations between derived equivalences on $3$-folds to this setting, for instance~\cite[Section 7.6]{DW1} which follows~\cite{Tod}.
\end{rem}

\section{Family constructions}\label{sect.fam}

I conclude with some straightforward constructions which realize each $n$-fold resolution as a family of $k$-fold resolutions for some $k<n$. I first explain the statement for $k=n-1$. This coincides with Proposition~\ref{prop.fam4} for $n=4$ and a case of Proposition~\ref{prop.famhigher} for $n=5$.

\begin{prop}\label{prop.iterate} For $n\geq 4$ we have for $i=1,\dots,n-2$
\[ X_i^{(n)}(V_1, \dots, V_{n-1}) \cong X_i^{(n-1)} (\pi_1, \dots, \pi_{n-2}) \]
where we take the following bundles.
\begin{align*} 
\lab{\pi_1} & V_1 \vecto \P V_{n-1} \\
\labpos{\vdots} & \\
\lab{\pi_{n-3}} & V_{n-3} \vecto \P V_{n-1} \\
\lab{\pi_{n-2}} & V_{n-2}\otimes\cO(-1) \vecto \P V_{n-1}
\end{align*}
\end{prop}
\begin{proof} We may write
\[ \lab{X_1^{(n)}} V_1 \otimes \cO(-1,\dots,-1,-1,-1)  \vecto (\P V_2 \times \dots \times \P V_{n-3}) \times \P V_{n-2} \times \P V_{n-1} \]
so that the claim again follows by the methods of Section~\ref{sect.fam_descrip}. Similar arguments suffice for the other $i$.
\end{proof}

\begin{rem} Note for completeness that the $n=3$ case of the above Proposition~\ref{prop.iterate} is true too, when suitably interpreted. For this, we take $X_1^{(2)}(V_1)$ as simply~$V_1$, because $Z^{(2)}(V_1) = V_1$ so no resolution is needed here. By extension we take $X_1^{(2)}(\pi_1)$ as simply $\pi_1$. We may then put 
\[ \lab{\pi_1} V_1 \otimes \cO(-1)  \vecto \P V_2 \]
so that the statement follows by definition of $X_1^{(3)}(V_1,V_2)$.
\end{rem}

By iterating the above Proposition~\ref{prop.iterate} in families, we may realize $X_1^{(n)}$ as a family of $X_1^{(k)}$. The following is a direct construction to show this fact, which coincides with the above when $k=n-1$.

\begin{prop}\label{prop.gen_fam} For $n>k\geq 3$ we have for $i=1,\dots,k-1$
\[ X_i^{(n)}(V_1, \dots, V_{n-1}) \cong X_i^{(k)} (\pi_1, \dots, \pi_{k-1} ) \]
where we take the bundle
\[ \lab{\pi_{k-1}} V_{k-1} \otimes\cO(-1,\dots,-1) \vecto \P V_k \times \dots \times \P V_{n-1} \]
and $\pi_1,\dots,\pi_{k-2}$ bundles with constant fibre $V_1,\dots,V_{k-2}$ over the same base.
\end{prop}
\begin{proof} Note first for consistency that the dimension of the base is $n-k$. Observe that, taking $k\geq 4$, we may write the following.
\begin{multline*} \lab{X_1^{(n)}} V_1 \otimes \cO(-1,\dots,-1,-1,-1,\dots,-1)  \\ \vecto (\P V_2 \times \dots \times \P V_{k-2}) \times \P V_{k-1} \times (\P V_k \times \dots \times \P V_{n-1}) \end{multline*}
We deduce the claim for $i=1$ for $k\geq 4$. A similar argument gives the claim for other $i$, and also $k= 3$.
\end{proof} 



\end{document}